\documentclass[english]{elsarticle}
\usepackage[T1]{fontenc}
\usepackage[latin9]{inputenc}
\usepackage{amsthm}
\usepackage{amsmath}
\usepackage{amssymb}
\usepackage{esint}

\makeatletter
\numberwithin{equation}{section}
\numberwithin{figure}{section}
\theoremstyle{plain}
\newtheorem{thm}{\protect\theoremname}
  \theoremstyle{plain}
  \newtheorem{prop}[thm]{\protect\propositionname}
  \theoremstyle{plain}
  \newtheorem{lem}[thm]{\protect\lemmaname}
  \theoremstyle{definition}
  \newtheorem*{example*}{\protect\examplename}

\makeatother

\usepackage{babel}
  \providecommand{\examplename}{Example}
  \providecommand{\lemmaname}{Lemma}
  \providecommand{\propositionname}{Proposition}
\providecommand{\theoremname}{Theorem}

\begin{document}

\title{Tail Behavoir of sums of random components}

\author{Yu Li}

\address{Faculté des Sciences, de la Technologie et de la Communication 6, }

\address{rue Richard Coudenhove-Kalergi L-1359 }

\address{Luxembourg}
\begin{abstract}
A class of stochastic processes strongly related to random sums plays
an important role in network and in finance. In this paper we study
this kind of stochastic process discuss an overtime unchanged parameter
and reveal its asymptotic behavior. 
\end{abstract}
\maketitle

\section{\label{sec:Introduction}Introduction}

A class of discrete stochastic processes $\left\{ X_{n}\right\} $
of the form 
\begin{equation}
X_{n+1}=\sum_{i=0}^{X_{n}}\xi_{i}^{(n+1)}\label{eq:definition}
\end{equation}
plays an important role in modeling of network \citep*{delay_of_interconnected_flows}
and in modeling of finance\citep{modeling_of_extremal_events}. In
this paper we study this kind of stochastic process, introduce an
overtime unchanged parameter to characterize it and reveal its asymptotic
behavior. As we see in (\ref{eq:definition}), the interesting thing
is that here $X_{n}$ is the counting process for its successor $X_{n+1}$
for each $n$.

Our main motivation comes from insurance. We recall the classical
model of insurance risk \citep{modeling_of_extremal_events}. Let
$u$ stand for initial capital, $c$ for loaded premium rate and $N(t)$
stands for the number of claims until time $t$. The total claim amount
$S_{t}$ consists of a random sum of independent and identically distributed
claims $X_{i}$

\[
U(t)=u+ct-S_{t},S_{t}=\sum_{i=1}^{N(t)}X_{i},t\geq0
\]
. It is common to simplify this model further by assuming $N(t)$
is a homogeneous Poisson process independent of $X_{i}$

As a motivation, we recall a model of network traffic proposed in
\citep{delay_of_interconnected_flows}. Let us consider a chain of
routers through which we send two packets. Here we study inter delay
time $\tau_{n}$, time interval between observed packets. We fix a
basic time interval, assuming to be $1$ which corresponds to a single
service time on each router. If the inter delay time is $k$, following
our assumption, there are $k-1$ packets between the two observed
packets. The inter delay time can be changed in passing the routers
influenced by lateral traffic. Here let random variables $\xi_{i}^{(n)}$
represent the lateral traffic and it can be interpreted as ``the
sum of the packets entering this chain'' while the $i$-th packet
is being served. After the two observed packets pass $n+1$-th router,
the inter delay time is of the form
\[
\tau_{n+1}=\sum_{i=1}^{\tau_{n}}\xi_{i}^{(n)}
\]
Consider, if no packets enters and leaves the chain on the $n$-th
router, all $\xi_{i}^{n}$ take value 1 then $\tau_{n+1}=\tau_{n}$.
If only one packet leave the chain, then $\tau_{n+1}=\tau_{n}-1$.
So if we set $\xi_{i}^{n}\geq0$ , then both ``entering'' and ``leaving''
cases are included.

A major goal is to show the limit distribution. Since it is not possible,
in general, to find the explicit form of the limit distribution, we
will one of most interesting features is the behavoir of the tails.

In this paper we study a class of stochastic processes related to
random sum. We will see that our work describing the behavoir of inter
delay time, is a generalization of that of U. Sorger and Z. Suchanecki\citep{delay_of_interconnected_flows}
which has heavy tails. 

First, we introduce a parameter, which is associated with tail property
of a distribution, to characterize this stochastic process. We prove
that this parameter is convolution invariant, unchanged over time
and it depends only on components random variables. As an extra Bonus
we found that for convoluted distribution $F$ the term $\overline{F}(x)e^{xs}$
can never converge to a positive constant, it can only be divergent
or converge to $0$. 

Second, we study the asymptotic property of such stochastic process
$\left\{ X_{n}\right\} $ and present a way to calculate its limiting
distribution.

As mean and variance depend on $n$, $\left\{ X_{n}\right\} $ is
not a stationary stochastic process. However, there is a parameter
which can be regarded as a numerical characteristic of this stochastic
process. Such a parameter will be handled in section \ref{sec:A-convolution-invariant}.
We will see that such a parameter can deal with not only discrete
case, but also continuous case.

\section{\label{sec:Basic-properties-of}Basic properties of this class of
stochastic process}

First we give a precise definition of this stochastic process $\left\{ X_{n}\right\} $.
We consider a class of stochastic process $X_{n}$ equipped with a
given stochastic process $\xi_{i}^{(n)}$ of the form
\begin{equation}
X_{n+1}=\sum_{i=0}^{X_{n}}\xi_{i}^{(n+1)}\label{eq:insider_random_sum}
\end{equation}
such that
\begin{enumerate}
\item $X_{0}$ is a positive integer constant
\item $X_{n}$ takes values in set of positive integers
\item $\xi_{i}^{(n)}$ is a given independent and identically distributed
stochastic process
\item $X_{n}$ and $\xi_{i}^{\left(k\right)}$ are independent for all $n,k,i$.
\end{enumerate}
(\ref{eq:insider_random_sum}) implies a compound distribution \citep{Adelson},
so we have so following properties
\begin{prop}
For this stochastic process $\left\{ X_{n}\right\} $ we have\end{prop}
\begin{enumerate}
\item $E\left[X_{n+1}\right]=E\left[X_{n}\right]E\left[\xi\right]$
\item $\sigma^{2}\left[X_{n+1}\right]=E\left[X_{n}\right]\sigma^{2}\left[\xi\right]+\sigma^{2}\left[X_{n}\right]E\left[\xi\right]^{2}$
\item $F_{n+1}\left(y\right)=\sum_{k=1}^{+\infty}P\left[X_{n}=k\right]F_{\xi}^{*k}\left(y\right)$
\end{enumerate}
where $F_{n}$ denotes the distribution of $X_{n}$ and $F_{\xi}$
the distribution of $\xi.$
\begin{proof}
See \citep{Adelson}.
\end{proof}
Both parameters, mean and variance, depend on $n$, so $\left\{ X_{n}\right\} $
is not stationary. However, there is a parameter unchanged over time,
which will be studied in the next section.

\section{\label{sec:A-convolution-invariant}A convolution invariant Parameter}

As we mention in section\ref{sec:Basic-properties-of} such stochastic
process $\left\{ X_{n}\right\} $ is not stationary. However, we can
characterize this stochastic process with a parameter. In this section
we study such a parameter. We define a parameter by
\begin{equation}
C\left(F\right):=\sup\left\{ t\in\mathbb{R}_{\geq0}:\lim_{x\rightarrow+\infty}\overline{F}(x)e^{xt}=0\right\} \label{eq:convolutioninvariant}
\end{equation}
where $F$ denotes a given distribution of a random variable and tail
$\overline{F}=1-F$. First, we prove this parameter is convolution
invariant and second, we prove this parameter characterize this stochastic
process. 
\begin{thm}
The parameter $C$ is the limit of hazard rate function, i.e.
\[
C(F)=\lim_{x\rightarrow\infty}\hat{m}(x)
\]
\[
\]
It is not difficult to see \citep{tailpro} 
\begin{equation}
\overline{F^{*2}}\left(x\right)=\overline{F}(x)+\int_{0}^{x}\overline{F}(x-y)dF(y)\label{eq:tailconvolution}
\end{equation}
What's more, we have a more general form
\begin{equation}
\overline{F^{*\left(k+1\right)}}\left(x\right)=\overline{F^{\left(k\right)}}(x)+\int_{0}^{x}\overline{F^{\left(k-1\right)}}(x-y)dF(y),\forall k\in\mathbb{N}\label{eq:tailconvolution2}
\end{equation}

\end{thm}
Before we prove that the parameter $C$ in (\ref{eq:convolutioninvariant})
depending on a distribution $F$ is convolution invariant, we need
the following lemma.
\begin{lem}
\label{lem:lemma1}If $\lim_{x\rightarrow+\infty}\overline{F}(x)e^{xs}=0$,
then there exists positive $s^{*}$ such that 
\[
\overline{F}(x)\leq e^{-xs^{*}}
\]
as long as $x>X$ for some $X$.\end{lem}
\begin{proof}
From the definition of convergence we have $\forall\epsilon>0,\exists X>0,\forall x>X$
such that
\[
0<\overline{F}(x)e^{xs}<\epsilon
\]
We choose $\epsilon<1$ and $s*$ fixed by 
\[
s^{*}=s+\frac{\ln\frac{1}{\epsilon}}{X}
\]
then
\[
\overline{F}(x)\leq\epsilon e^{-xs}\leq\epsilon^{\frac{x}{X}}e^{-xs}=e^{-x\left(s+\frac{\ln\left(1/\epsilon\right)}{X}\right)}\leq e^{-xs^{*}}
\]
The lemma is proved.\end{proof}
\begin{prop}
\label{pro1}If $\overline{F}(x)e^{xs}$ converges to $0$ for some
$s>0$, then $\overline{F^{*2}}(x)e^{xs}$ converges to $0$.\end{prop}
\begin{proof}
First we consider discrete case. We use the same $s^{*}$ and $X$
in lemma \ref{lem:lemma1} and let $x>2X$. Then we calculate (\ref{eq:tailconvolution})
and obtain
\begin{alignat*}{1}
\overline{F^{*2}}(x) & =\overline{F}(x)+\sum_{y=0}^{x-1}\overline{F}(x-y)\left[\overline{F}(y)-\overline{F}(y+1)\right]\\
 & =\overline{F}(x)+\sum_{y=0}^{x-1}\overline{F}(x-y)\overline{F}(y)\left[1-\frac{\overline{F}(y+1)}{\overline{F(y)}}\right]\\
 & \leq\overline{F}(x)+\sum_{y=0}^{X}\overline{F}(x-y)\overline{F}(y)+\sum_{y=X+1}^{x-X}\overline{F}(x-y)\overline{F}(y)\\
 & +\sum_{y=x-X}^{x-1}\overline{F}(x-y)\overline{F}(y)\\
 & \leq\overline{F}(x)+2\sum_{y=0}^{X}e^{-s^{*}\left(x-y\right)}+\left(x-2X\right)e^{-s^{*}x}\\
 & =\overline{F}(x)+e^{-s^{*}x}\left(\frac{2\left(e^{s^{*}X}-1\right)-e^{-s^{*}x}}{1-e^{-s^{*}}}+x-2X\right)
\end{alignat*}
Finally
\[
0\leq\overline{F^{*2}}(x)e^{xs}\leq\overline{F}(x)e^{xs}+\frac{1}{e^{\left(s^{*}-s\right)x}}\left(\frac{2\left(e^{s^{*}X}-1\right)-e^{-s^{*}x}}{1-e^{-s^{*}}}+x-2X\right)
\]
Due to L'Hôpital's rule the last term converges to $0$ as $x\rightarrow+\infty$.
Therefore $\overline{F^{*2}}(x)e^{sx}$ converges to $0$. 

Second we consider continuous case and let $\left\lfloor x\right\rfloor $
denote the integer part of $x,$ then

\begin{alignat*}{1}
\overline{F^{*2}}(x) & =\overline{F}(x)+\int_{0}^{x}\overline{F}(x-y)dF(y)\\
 & \leq\overline{F}(x)+\sum_{y=0}^{\left\lfloor x\right\rfloor -1}\overline{F}(x-y)\left[\overline{F}(y)-\overline{F}(y+1)\right]\\
 & \leq\overline{F}(x)+e^{-s^{*}\left\lfloor x\right\rfloor }\left(\frac{2\left(e^{s^{*}X}-1\right)-e^{-s^{*}\left\lfloor x\right\rfloor }}{1-e^{-s^{*}}}+\left\lfloor x\right\rfloor -2X\right)
\end{alignat*}
Analog, $\overline{F^{*2}}(x)e^{xs}$ converges to $0$. So far both
discrete case and continuous case are considered. This proposition
is proved. 
\end{proof}
In the following we prove a more general proposition.
\begin{prop}
\label{pro2}If $\overline{F}(x)e^{xs}$ converges to $0$ for some
$s>0$, then $\overline{F^{*k}}(x)e^{xs}$ converges to $0$ for any
natural number $k$.\end{prop}
\begin{proof}
We prove it by mathematical induction. Assume $ $$\overline{F^{*k}}(x)e^{xs}$
converges to $0$, what we need to show, due to (\ref{eq:tailconvolution2}),
is
\[
\lim_{x\rightarrow+\infty}e^{xs}\left[\sum_{y=0}^{x-1}\overline{F^{*k}}\left(x-y\right)\left(\overline{F}\left(y\right)-\overline{F}\left(y+1\right)\right)\right]=0
\]
 We know from lemma \ref{lem:lemma1}, there exist positive number
$s_{1}$ and $X_{1}$ such that for all

\[
\overline{F}\left(x\right)<e^{xs_{1}}
\]
as long as $x>X_{1}$. 

Analog there exist positive number $s_{2}$ and $X_{2}$ such that
for all

\[
\overline{F}\left(x\right)<e^{xs_{2}}
\]
as long as $x>X_{2}$. 

Let $X=\max\left\{ X_{1},X_{2}\right\} $ and $s^{*}=\min\left\{ s_{1},s_{2}\right\} $,
then
\begin{eqnarray*}
0 & \leq & e^{xs}\left[\sum_{y=0}^{x-1}\overline{F^{*k}}\left(x-y\right)\left(\overline{F}\left(y\right)-\overline{F}\left(y+1\right)\right)\right]\\
 & \leq & \frac{x-2X+2\left(\frac{e^{X}-1}{1-e^{s^{*}}}\right)}{e^{\left(s^{*}-s\right)x}}
\end{eqnarray*}
Again the last term converges to $0$ as $x\rightarrow+\infty$, hence
$\overline{F^{*\left(k+1\right)}}\left(x\right)e^{xs}$ converges
to $0$. The proposition is proved.
\end{proof}
Second we will prove a somehow surprising result that for a convoluted
distribution the term $\overline{F^{*k}}(x)e^{xs}$ can never converge
to a positive constant, which can be formulated in this proposition.
\begin{prop}
\label{pro3}If $\overline{F}(x)e^{xs}$ tends to a positive constant
$c$ for some $s$, then $\overline{F^{*2}}(x)e^{xs}$ tends to $+\infty$
for all $k\geq2$.\end{prop}
\begin{proof}
Let $\overline{F}(x)e^{xs}$ converges to a positive constant $C$
. We see two inequalities. For $\forall\epsilon>0,\exists X>0,\forall x>X$
the inequality
\[
C-\epsilon<e^{sx}\overline{F}(x)<C+\epsilon
\]
the inequality 
\[
\overline{F}(x)e^{xs}\geq\min\left\{ \overline{F}(0),\overline{F}(1)e^{1s},...,\overline{F}(X)e^{Xs},C-\epsilon\right\} 
\]
Let 
\[
m:=\min\left\{ \overline{F}(0),\overline{F}(1)e^{1s},...,\overline{F}(X)e^{Xs},C-\epsilon\right\} 
\]

What we need to prove is that
\[
\lim_{x\rightarrow\infty}\frac{\sum_{y=0}^{x-1}\overline{F}(x-y)\left[F(y)-F(y+1)\right]}{\overline{F}(x)}=+\infty
\]
 Then we estimate
\begin{eqnarray*}
\frac{\sum_{y=0}^{x-1}\overline{F}(x-y)\left[F(y)-F(y+1)\right]}{\overline{F}(x)} & \geq & \frac{\sum_{y=X+1}^{x-1}\overline{F}(x-y)\left[F(y)-F(y+1)\right]}{\overline{F}(x)}\\
 & \geq & \frac{\sum_{y=0}^{x-1}\overline{F}(x-y)e^{s\left(x-y\right)}\left[F(y)d^{sy}-\frac{F(y+1)e^{s\left(y+1\right)}}{e^{s}}\right]}{\overline{F}(x)e^{xs}}\\
 & \geq & \sum_{y=X+1}^{x-1}\frac{m\left[\left(1-\frac{1}{e^{s}}\right)C-\epsilon\left(1+\frac{1}{e^{s}}\right)\right]}{C+\epsilon}
\end{eqnarray*}
The last term tends to infinity as $x\rightarrow+\infty$. The proposition
is proved.
\end{proof}
Combining propositions (\ref{pro1}) (\ref{pro2}) and (\ref{pro3}),
we can obtain the following theorem
\begin{thm}
Parameter $C(F)$ defined through
\[
C\left(F\right):=\sup\left\{ t\in\mathbb{R}_{\geq0}:\lim_{x\rightarrow+\infty}\overline{F}(x)e^{xt}=0\right\} 
\]
is invariant under convolution, i.e. 
\[
C\left(F\right)=C\left(F^{*k}\right)
\]
for all natural number $k$.\end{thm}
\begin{example*}
We consider exponential distribution
\[
F(x)=1-e^{-\lambda x}
\]
and its convolution
\[
F^{*2}(x)=1-e^{-\lambda x}\left(1+\lambda x\right)
\]
Obviously
\[
C\left(F\right)=C\left(F^{*2}\right)=\lambda
\]
\end{example*}
\begin{thm}
The parameter of $X_{n}$ in (\ref{eq:insider_random_sum})depends
only on $\xi$, what's more, for all $n>0$
\[
C\left(F_{n}\right)=C\left(F_{\xi}\right)
\]
\end{thm}
\begin{proof}
From \citet{aconvolutioninvariantparameter} we have 
\[
\overline{F_{n}}(x)=\sum_{k=1}^{+\infty}P\left[X_{n-1}=k\right]\overline{F_{\xi}^{*k}}(x)
\]
then 
\[
\overline{F_{n}}(x)e^{xs}=\sum_{k=1}^{+\infty}P\left[X_{n-1}=k\right]\overline{F_{\xi}^{*k}}(x)e^{xs}
\]
If $\overline{F_{\xi}^{*k}}(x)e^{xs}$ converges to $0$ for some
$s>0$, $\overline{F_{n}}(x)e^{xs}$ converges 0. Hence this parameter
$C$ depends only on $\xi$ while is independent of $n$.
\end{proof}

\section{Asymptotic properties}

In this section we study the asymptotic properties of $\left\{ X_{n}\right\} $
and denote the limiting probability function by $f^{\star}$. 

We recall

\begin{equation}
F_{n+1}\left(y\right)=\sum_{k=1}^{+\infty}P\left[X_{n}=k\right]F_{\xi}^{*k}\left(y\right)\label{eq:convolution1}
\end{equation}
and further
\begin{equation}
F_{n+1}\left(y-1\right)=\sum_{k=1}^{+\infty}P\left[X_{n}=k\right]F_{\xi}^{*k}\left(y-1\right)\label{eq:convolution2}
\end{equation}
Letting (\ref{eq:convolution1})-(\ref{eq:convolution2}), then

\[
f_{n+1}\left(y\right)=\sum_{k=1}^{+\infty}f_{n}\left(k\right)f_{\xi}^{*k}\left(y\right)
\]
where $f_{n}$ denotes the probability function of $X_{n}$ and $f_{\xi}$
the probability function of $\xi.$

We consider $\left(f_{\xi}^{*k}\left(y\right)\right)_{y,k}$ as a
Markov operator, which maps a probability function to another, and
denote it by $M_{\xi}$, then the above expression can be rewritten
in form of
\[
f_{n+1}\left(y\right)=\left(M_{\xi}f_{n}\right)\left(y\right)
\]
and furthermore 
\[
f_{n+1}\left(y\right)=\left(M_{\xi}^{n+1}f_{0}\right)\left(y\right)
\]

So for the asymptotic property of $\left\{ X_{n}\right\} _{n}$, we
have the following proposition.
\begin{prop}
$f_{n}$ converges to a function $f^{\star}$ satisfying 
\begin{equation}
\sum_{k\neq j}f^{\star}\left(k\right)f_{\xi}^{*k}\left(j\right)+f^{\star}\left(j\right)\left(1-f_{\xi}^{*j}\left(j\right)\right)=0\label{propositionconvergences}
\end{equation}
\end{prop}
\begin{proof}
Operator $M_{\xi}$ is of the form 
\[
M_{\xi}=\sum_{i=1}\lambda_{i}E_{\lambda_{i}}
\]
and further
\begin{equation}
M_{\xi}^{n}=\sum_{\lambda_{i}\in\sigma\left(M_{\xi}\right)}\lambda_{i}^{n}E_{\lambda_{i}}\label{eq:convergence}
\end{equation}
where $\lambda_{i}$ denote the eigenvalues of $M_{\xi}$ and $E_{i}$
denote the projections projecting the entire space into eigenspaces
corresponding to $\lambda_{i}$.

We apply Gelfand theorem \citep{linearoperators} to here
\[
\sigma\left(M_{\xi}\right)=\lim_{n\rightarrow\infty}\left\Vert M_{\xi}^{n}\right\Vert ^{\frac{1}{n}}=1
\]
In other words, $\left|\lambda_{i}\right|$ is bound by $1$ for all
eigenvalues.

Let $n$ tend to infinity and we see that only eigenvalue $\lambda_{i}=1$
contributes in the right hand side of (\ref{eq:convergence}), therefore
this limiting probability function $f^{\star}$ is a fix-point of
$M_{\xi}$, i,e. 
\[
M_{\xi}f^{\star}=f^{\star}
\]

After simplification we obtain (\ref{propositionconvergences}) and
the proposition is proved.
\end{proof}

\section{Conclusion}

In this paper we study a class of stochastic processes, which is widely
used in modeling of Network and Finance and obtain its two asymptotic
properties.

Our main aim 

\bibliographystyle{plain}
\nocite{*}
\bibliography{stochasticprocess}

\end{document}